\documentclass[12pt]{amsart}

\usepackage{amsmath}
\usepackage{amsfonts}
\usepackage{amssymb}

\usepackage{hyperref}

\DeclareMathOperator{\T}{\mathbb{T}}
\DeclareMathOperator{\Z}{\mathbb{Z}}
\DeclareMathOperator{\Schur}{Sch}
\DeclareMathOperator{\Riesz}{Rsz}

\newtheorem{theorem}{Theorem}[section]
\newtheorem{lemma}[theorem]{Lemma}

\theoremstyle{definition}

\theoremstyle{remark}
\newtheorem{remark}[theorem]{Remark}

\numberwithin{equation}{section}

\begin{document}

\def\now{26 June 2024}

\title[\now \hfill
Lacunary coefficients of~$H^1$-functions.
\hfill
]
{The missing proof of Paley's theorem\\
about lacunary coefficients}
\author{John J.F. Fournier}
\address{Department of Mathematics\\
University of British Columbia\\
1984 Mathematics Road\\
Vancouver BC\\
Canada V6T 1Z2}
\email{fournier@math.ubc.ca}
\thanks{Partially announced
at the 6th Conference on Function Spaces in May 2010.}
\thanks{Research supported by NSERC grant 4822.}


\dedicatory{To the memory of Frank Forelli, who set me on this path.}

\date{\now}

\begin{abstract}
We modify the
classical
proof of Paley's
theorem about lacunary
coefficients of functions in~$H^1$
to work without analytic factorization.
This leads to the first direct proof of the extension
of Paley's theorem that we applied
to the former Littlewood conjecture about~$L^1$ norms of
exponential sums.
\end{abstract}

\maketitle

\markleft{
John J.F. Fournier
}
\pagestyle{myheadings}

\section{Introduction}\label{sec:intro}

Given an integrable function~$f$
on the interval~$(-\pi, \pi]$, form its
Fourier coefficients
\[
\hat f(n) = \frac{1}{2\pi}\int_{-\pi}^\pi
f(t)e^{-int}\,dt.
\]
Use the same measure~$(1/2\pi)\,dt$
in computing~$L^p$ norms.
Call a set of nonnegative integers
\emph{strongly lacunary}
if it is the range
of a sequence,~$(k_j)$ 
say, with the property that
\begin{equation}\label{eq:stronglylacunary}
k_{j+1} > 2k_j
\quad\textnormal{for all $j$.}
\end{equation}

In Section~\ref{sec:two-step}, we give a new
proof of the following statement.

\begin{theorem}\label{th:Paley}
There is a constant~$C$ so that
if~$K$ is strongly lacunary,
and if $\hat f(n) = 0$ when~$n<0$, then
\begin{equation}\label{eq:basicestimate}
\left[\sum_{k \in K} |\hat f(k)|^2\right]^{1/2}
\le C\|f\|_1.
\end{equation}
\end{theorem}

Paley's proof~\cite{REACH1}
of this
used
``analytic''
factorization
of such functions~$f$ as products
of two measurable functions
with the same absolute value
and
with Fourier coefficients that
also vanish at all negative integers.
We use factors with the same absolute value,
but we do not
require that their coefficients vanish anywhere.
Paley's proof used orthogonal projections
of~$L^2$ onto subspaces determined by the set~$K$.
We use subspaces that
may
also depend
on the choice of factors.

This
allows
us to give the first direct proofs of some refinements,
stated here as Theorems~\ref{th:S(K)} and~\ref{th:S_1(K)},
of Paley's theorem.
They
were proved
in
a dual way
in~\cite{FouArk},
and
used
there
to give a new proof of
``half'' of
the Littlewood conjecture
about~$L^1$ norms
of exponential sums.

We
prove
Paley's theorem in the next section.  In Section~\ref{sec:ordered}, we
extend
this to compact abelian groups with partially-ordered duals.  We use Riesz products in Section~\ref{sec:totalize} to deduce some of these extensions from
previously-known
results for totally-ordered dual groups.  In Section~\ref{sec:fewer}, we  show that our new method works with weaker
hypotheses.
We
weaken those further in Section~\ref{sec:completing}, using Riesz products again.
Finally,
in an appendix,
we
examine the relation between
the method in this paper
and
the one that was applied to Paley's theorem in~\cite{KP}.

\begin{remark}
\label{Operator-valued}
The functions in this paper are scalar-valued.
Our
methods
are applied
to some operator-valued 
functions in~\cite{FouArXiv},
and yield a new proof of the main result
in~\cite{LP}.
\end{remark}

\section{Pairs of nested projections}
\label{sec:two-step}

\begin{proof}[Proof of Paley's theorem]
When~$f$ satisfies the hypotheses of Theorem~
\ref{th:Paley},
factor~$f$ as~$g\overline h$,
where~$g$ and~$h$ are
measurable,
and~$|g| = |h|$.
%
Let~$z$
denote
the exponential function
mapping each number~$t$
in the interval~$(-\pi, \pi]$ to~$e^{it}$.
Then~$g$ and the products~$z^n h$
belong to~$L^2(\pi, \pi]$. Consider the inner products
\begin{equation}\label{eq:coefficients}
(g, z^n h)
= \frac{1}{2\pi}\int_{-\pi}^{\pi}
g(t)\overline{h(t)}e^{-int}\,dt
= \hat f(n).
\end{equation}

It suffices to prove
inequality~\eqref{eq:basicestimate}
when the set~$K$ is the range of a finite
increasing sequence~$(k_j)_{j=1}^J$.
Let~$A_j$ be the operator
on~$L^2(-\pi, \pi]$
that multiplies each function
by~$z^{k_j}$.
Then
\begin{equation}
\label{eq:InnerProducts}
(g, A_jh)
= \hat f(k_j).
\end{equation}
This reduces matters
to showing that
there is a constant~$C$ so that
\begin{equation}\label{eq:HilbertEstimate}
\left[\sum_{j=1}^J
|(g, A_jh)|^2\right]^{1/2}
\le C\|g\|_2\|h\|_2\quad\textrm{for all~$J$.}
\end{equation}

Let~$L_j$ be the closure
in~$L^2$ of the subspace spanned by the
products~$z^n h$
in which~$n < -k_j$.
Let~$P_j$
project~$L^2$ orthogonally onto~$L_j$.
These projections form a decreasing
nest as~$j$ increases.

Also consider
the subspaces~$A_jL_j$
and~$A_{j+1}L_j$, where~$j < J$ for the latter.
Every image~$A_jL_j$ is the closure in~$L^2$ of the span of the products~$z^n h$ for which~$n < 0$. By formula~(\ref{eq:coefficients})
and the hypothesis that~$\hat f(n) = 0$ for all~$n < 0$,
the function~$g$ is orthogonal
to~$A_jL_j$
for all~$j$.

The image~$A_{j+1}L_j$
is the closure of the span
of the products~$z^n h$
for which~$n < k_{j+1} - k_j$.
Strong
lacunarity
is equivalent to having
\begin{equation}\label{eq:altHadam}
k_{j} < k_{j+1} - k_{j}.
\end{equation}
It follows that
\begin{equation}
\label{eq:Membership}
A_jh \in A_{j+1}L_j
\quad\text{when~$j < J$.}
\end{equation}
Since~$k_j - k_{j-1} \le k_j< k_{j+1}-k_j$, the subspaces~$A_{j+1}L_j$ increase
as~$j$ increases.
Let~$Q_{j}$ project
orthogonally onto~$A_{j+1}L_j$ when~$1 \le j < J$;
let~$Q_0 = 0$, and~$Q_{J} = I$.
The projections~$Q_j$ form an increasing nest.

For each~$j$, these
choices and the membership condition~\eqref{eq:Membership}
make
\[
A_jh = Q_jA_jh
\quad\text{and}\quad
(g, A_jh) = (g, Q_{j}A_jh) = (Q_{j}g, A_jh).
\]
Rewrite the latter in the form
\begin{equation}
\label{eq:TwoStep}
(\{Q_j - Q_{j-1}\}g,A_jh)
+ (Q_{j-1}g, A_j h)
=
a_j + b_j
\quad\text{say.}
\end{equation}
By the
Cauchy-Schwarz inequality, the fact
that the operators~$A_j$ are contractions,
and the nesting of the projections~$Q_j$,
\begin{equation}\label{eq:outer}
\|(a_j)\|_2^2
\le \left(\sum_{j=1}^J\|\{Q_j - Q_{j-1}\}g\|_2^2\right)\|h\|_2^2
\le \|g\|_2^2\|h\|_2^2.
\end{equation}

Now~$b_1 = 0$,
because~$Q_0 = 0$.
Since~$A_j$ is unitary,~$A_jP_{j-1} =
Q_{j-1}A_j$ when~$j > 1$,
and then
\begin{equation}
\label{eq:b's}
b_j
= (g, Q_{j-1}A_j h)= (g, A_jP_{j-1} h).
\end{equation}
The fact that~$g\perp A_jL_j$
then
makes~$(g, A_jP_jh) = 0$,
and it follows that~$b_j =
(g, A_j\{P_{j-1} - P_j\}h)$.
So~$\|(b_j)\|_2^2 \le \|g\|_2^2\|h\|_2^2$ 
too, and
inequality~(\ref{eq:HilbertEstimate})
holds
with~$C=2$.
\end{proof}

The refinements
in Theorems~\ref{th:S(K)} and~\ref{th:S_1(K)} below were
proved in~\cite{FouArk}
using dual methods.
Here,
Theorem~\ref{th:S(K)} will follow from
an
analysis of the 
direct
proof above.
The notions 
in
the next two sections will then allow us to deduce
Theorem~\ref{th:S_1(K)}.

\begin{remark}
\label{rm:traditional}
To organize
Paley's proof
in the same way,
require that the factors~$g$ and~$\overline h$ both be analytic.
Replace
the subspaces~$L_j$ above
with
the closures in~$L^2$ of the
spans
of the functions~$z^n$ for which~$n < -k_j$;
in
many
cases,
these subspaces
are larger than the
ones
used above, 
but they nest as before,
as do
their images~$A_{j+1}L_j$,
which are generated by the functions~$z^n$ for which~$n < k_{j+1}-k_j$. Assuming that~$\hat g(n) = 0$ for all~$n < 0$ guarantees
that~$g \perp A_jL_j$ for all~$j$,
because~$A_jL_j$
is generated by
the function~$z^n$ for which~$n < 0$. 
Assuming that~$\overline h$ is analytic makes~$\hat h(n) = 0$ for all~$n > 0$.
It follows that~$A_jh \in A_{j+1}L_j$
when~$j < J$, and the rest of the proof above applies.
\end{remark}



\begin{remark}
The subspaces~$L_j$ that we 
used in the proof of Theorem~\ref{th:Paley} are invariant under multiplication by~$\overline z$, and their conjugates are invariant under multiplication by~$z$.
In~\cite{Forelli}, 
it
was
observed
that
those conjugate subspaces must be simply 
invariant when~$f$ is analytic, 
and 
the characterization of simply invariant subspaces of~$L^2(\T)$
was then used
to show that both factors~$\overline h$ and~$g$ of~$f$ can be chosen to be analytic too.
\end{remark}

\section{Partially ordered dual groups}\label{sec:ordered}

Our proof of Theorem~\ref{th:Paley}
resembles the one given in~\cite[Section~2]{FouPac} for the following statement,
which differs
only in the set where~$\hat f$
is required to vanish.

\begin{theorem}\label{th:complementary}
There is a constant~$C$ so that
if~$K$ is strongly lacunary,
and if~$\hat f(n) = 0$ for all positive integers~$n$ lying outside~$K$, then
\begin{equation}
\label{eq:complementaryestimate}
\left[\sum_{k \in K} |\hat f(k)|^2\right]^{1/2}
\le C\|f\|_1.
\end{equation}
\end{theorem}

Various other
methods
in~\cite[p.~533--4]{Mey},
~\cite{Vino} and~\cite[Theorem~10]{FouPac}
derive
this
conclusion
from weaker
conditions on~$K$
or~$f$.
In Remark~\ref{rm:ComplementaryProof} below,
we outline our
direct
proof of Theorem~\ref{th:complementary}.
That proof extended 
to compact abelian groups with
partially
ordered duals.

As in~\cite[Section~8.1]{RudGp},
where the dual group~$\Gamma$ is written additively,
total orders arise when
there is an additive semigroup~$P$
with the two properties
\begin{equation}\label{eq:total}
P\cap (-P) = \{0\},
\qquad
P\cup(-P) = \Gamma.
\end{equation}
We
then
write
that~$\gamma \le \gamma'$ when~$\gamma'-\gamma \in P$.
Partial orders arise
in
the same
way
when
the nonnegative cone~$P$ need only satisfy the first condition above.
We now confirm that our new proof of Paley's theorem
extends
to that setting.


Call a subset~$K$ of~$P$
\emph{strongly lacunary}
if for each pair~$\gamma$ and~$\gamma'$ of distinct members of~$K$,
one of the differences~$\gamma - 2\gamma'$ or~$\gamma' - 2\gamma$
belongs to the strictly positive
cone~$P' = P\backslash\{0\}$.
%
The following extension of
Theorem~\ref{th:Paley} is
known~\cite[Section~8.6]{RudGp},
with a different proof,
in the cases where
the partial order on~$\Gamma$
is a total order.

\begin{theorem}\label{th:partialPaley}
There is a constant~$C$
with the following property.
Let~$G$ be a compact abelian group
with a partially ordered dual~$\Gamma$.
Let~$K$ be strongly lacunary
relative to that order.
If~$f \in L^1(G)$,
and~$\hat f(\gamma) = 0$ for all characters~$\gamma$
in the strictly negative cone~$-P'$, then
\begin{equation}\label{eq:partPaley}
\left[\sum_{\gamma \in K} |\hat f(\gamma)|^2\right]^{1/2}
\le C\|f\|_1.
\end{equation}
\end{theorem}

\begin{proof}
Without loss of generality,~$K$ is finite. Enumerate it
in increasing order as~$(\gamma_j)_{j=1}^J$.
Factor~$f$ measurably as~$g\overline h$ with~$|g| = |h|$.
Make the following choices for each~$j$.
Let~$A_j$ be the operator that multiplies
each function
in~$L^2(G)$
by~$\gamma_j$.
Let~$L_j$ be the closure
in~$L^2(G)$
of the subspace spanned
the products~$\gamma h$
in which~$\gamma < -\gamma_j$.
Define the nested projections~$P_j$ and~$Q_j$ as before,
and split the inner product~$(g, A_jh)$ in the same way
to get inequality~(\ref{eq:partPaley}) with~$C=2$.
\end{proof}

\begin{remark}\label{rm:Complementary}
\label{rm:ComplementaryProof}
In~\cite
{FouPac},
we proved Theorem~\ref{th:complementary}
using the same
factorization and the same operators~$A_j$ as in our proof of
Theorem~\ref{th:Paley},
but
using
the
subspaces~$M_j$ spanned by the products~$z^nh$
in which~$-k_j \le n < 0$.
Those subspaces form an increasing nest, as do their images~$A_jM_j$. We used orthogonal 
projections,~$P'_j$
and~$Q'_j$ say,
with ranges~$M_j$ and~$A_jM_j$ respectively, 
also
letting~$P'_0 = 0$,
and~$Q'_{J+1} = I$.
When~$j < J$, the subspace~$A_{j+1}M_j$ is spanned by the products~$z^mh$ for which~$m$ lies in the
half-open
interval~$[k_{j+1} - k_j, k_{j+1})$.  By strong lacunarity, these integers~$m$ all fall in the gap between~$k_j$ and~$k_{j+1}$.
Since~$\hat f$ vanishes in 
these
gaps,~$g \perp A_{j+1}M_j$
for all~$j < J$. Also,~$A_jh \in A_{j+1}M_{j+1}$
for these values of~$j$.
So~$(g, A_jh)$ splits
here
as~$a'_j + b'_j$,
where
\begin{equation}
\label{eq:a' and b'}
a'_j
=
(\{Q_{j+1}-Q_j\}g, A_jh),
\quad\text{and}\quad
b'_j
= (Q_jg, A_jh).
\end{equation}
Then~$b'_j = (g, A_jP_jh)$.
This
can be rewritten
as~$(g, A_j\{P'_j - P'_{j-1}\}h)$,
because~$P'_{j-1} = 0$
when~$j=1$,
and~$g$
is orthogonal
to~$A_jM_{j-1}$
in the remaining cases.
Estimate~$\ell^2$ norms as 
above.
\end{remark}

\begin{remark}
The proof just above was derived from
Paley's proof of his
Theorem~\ref{th:Paley},
but it no longer worked for
that
theorem. Our new proof of the latter
resulted from a study of
the
argument
in Remark~\ref{rm:Complementary}
and the proof,
using analytic factorization
and projections onto finite-dimensional subspaces,
of the version of
Theorem~\ref{th:Paley}
in~\cite{KP}.
See
Appendix~\ref{sec:compare}
for more about the latter proof.

\end{remark}


\begin{remark}
The dual method in~\cite{FouArk} shows that the best constant in Theorems~\ref{th:Paley}
and~\ref{th:partialPaley}
is~$\sqrt2$.
The dual method in~\cite{Clu}
and~\cite{FouPAMS} 
shows that the best constant in Theorem~\ref{th:complementary}
is at most~$\sqrt e$.
Theorem~\ref{th:partialPaley}
also
follows, with constant~$2$,
by the dual method
in~\cite{PS} and~\cite{Smith}.
\end{remark}

\begin{remark}
In Theorem~\ref{th:partialPaley}, the set where the coefficients are required to vanish is no larger than a half space. Other methods~\cite{Ob} work when that set is significantly larger than a half 
space, and
yield inequality~\eqref{eq:partPaley} for more sets~$K$.
\end{remark}

\section{Finite Riesz products}
\label{sec:totalize}

We
consider
Fourier coefficients of certain measures
in
the proof of
Theorem~\ref{th:S_1(K)}.
We
confirm here that
Theorem~\ref{th:partialPaley}
extends
to regular Borel measures,
with the usual convention that
\[
\hat u(\gamma) = \int_G \overline{\gamma(x)}
\,d\mu(x),
\]
for such a measure~$\mu$.
We also show
how
Theorem~\ref{th:partialPaley}
follows in most cases of interest,
with a larger constant~$C$,
from its special case where the order is total.

Denote the total variation of~$\mu$ by~$\|\mu\|$.
Continue to work with a partial order on~$\Gamma$.
Suppose throughout this section that~$\hat u$ vanishes
on the strictly negative cone~$-P'$.

Given a finite subset~$K$ of~$\Gamma$,
let~$K' = K \backslash\{0\}$.
Recall some properties of
the product
\[
R_K := \prod_{\gamma' \in K'}
\left(1 + \frac{\gamma' + \overline{\gamma'}}{2}\right).
\]
 of nonnegative
 factors.  It expands
 as
 as
 \[
 \sum_\gamma
 c(\gamma)\gamma
 \]
in
which~$c(\gamma) \ne 0$
only when~$\gamma = \prod_{\gamma'\in K'}
(\gamma')^{\varepsilon_{\gamma'}}$,
where~$\varepsilon_{\gamma'} \in \{-1, 0, 1\}$
in all cases.
In the additive notation for~$\Gamma$,
\begin{equation}\label{eq:trimmed}
\gamma = \sum_{\gamma'\in K'}
\varepsilon_{\gamma'}
\gamma'.
\end{equation}
Denote the set of such characters~$\gamma$ by~$\Riesz(K)$; this includes the identity element~$0$ of~$\Gamma$, written as the empty sum.
Then
\begin{itemize}
\item{}
Each member~$\gamma$ of~$K'$ has a representation~\eqref{eq:trimmed}
with~$\varepsilon_{\gamma} = 1$
and with~$\varepsilon_{\gamma'} = 0$ otherwise.
\item{}
$c(\gamma) = 1/2$ if
there are no other
representations of~$\gamma$.
\item{}
$c(\gamma) > 1/2$ if there are other representations of~$\gamma$.
\end{itemize}
Similarly,~$c(0) \ge 1$.

Now assume that~$K$ is strongly lacunary. Then
\begin{itemize}
\item{}
$\Riesz(K) \subset P\cup(-P')$.
\item{}
The only representation~(\ref{eq:trimmed})
of~$0$
is the empty sum.
\end{itemize}
Hence~$c(0) = 1$.
Since~$\widehat{R_K} = c$, it
vanishes off~$P\cup(-P)$,
while
\begin{equation}\label{eq:RieszHat}
\widehat{R_K}(0) = 1,
\quad\text{and}\quad
\widehat{R_K}(\gamma) \ge \frac{1}{2}
\quad\text{when}\quad
\gamma \in K'.
\end{equation}
Since~$R_K \ge 0$,
\begin{equation}\label{eq:RieszNorm}
\|R_K\|_1 = \widehat{R_K}(0) = 1.
\end{equation}

Let~$f_K = \mu*R_K$.
Then~$\widehat{f_K} = \hat\mu\widehat{R_K}$, which
vanishes
on~$-P'$
because~$\hat\mu$ does.
Also,
\[
\left|\widehat{f_K}(\gamma)\right|
= \left|\hat\mu(\gamma)\widehat{R_K}(\gamma)\right|
\ge \frac{1}{2}\left|\hat \mu(\gamma)\right|
\quad\text{for all~$\gamma$ in~$K$.}
\]
Applying
Theorem~\ref{th:partialPaley}
to~$f_K$
yields
that
\begin{equation}
\label{eq:PaleyForMeasures}
\|\hat \mu|K\|_2
\le 2\left\|\widehat{f_K}|K\right\|_2
\le 4\|f_K\|_1
\le 4\left\|R_K\right\|_1\|\mu\|
=
4\|\mu\|.
\end{equation}

\label{rm:FromTotalOrder}
In many cases,
the partial order on~$\Gamma$
extends
to a total order. That is, the cone~$P$ imbeds in a
cone~$\tilde P$
which
satisfies both conditions in line~(\ref{eq:total}).
Then the set~$K$ is
strongly lacunary relative
to~$\tilde P$.

As noted above,~$\widehat{f_K}$
vanishes on~$-P'$.
Because of its factor~$\widehat{R_K}$,
it also vanishes off~$\Riesz(K)$,
and hence off~$P\cup(-P')$.
So~$\widehat{f_K}$
vanishes off~$P$,
and hence
off the larger
set~$\tilde P$.

Theorem~\ref{th:partialPaley} is already known
for the total order given by~$\tilde P$, and
yields that
$
\left\|\widehat{f_K}\vert K\right\|_2 \le C\|f_K\|_1.
$
It follows as above that
\begin{equation}
\label{eq:ForMeasures}
\|\hat \mu\vert K\|_2
\le 2C\|\mu\|.
\end{equation}

\begin{remark}

In the same cases, the version of Theorem~\ref{th:complementary} for partial orders follows 
as above
from the instance of it for total orders,
which has other proofs.

\end{remark}

\begin{remark}
\label{rm:notKnow}
We do not know how to use the method
above
to prove
Theorem~\ref{th:S(K)} below,
but
it will allow us
to
then
deduce
Theorem~\ref{th:S_1(K)}.
\end{remark}

\begin{remark}
\label{rm:To trig polys}
Replacing the Riesz product~$R_K$
above
with
a suitable sequence of
the trigonometric polynomials
discussed
in~\cite{BPR}
gives
the
part~$\|\hat \mu\vert K\|_2
\le 4\|\mu\|$
of
inequality~\eqref{eq:PaleyForMeasures}
with the constant~$4$ replaced by~$2$.
The
use of finite Riesz products to pass from more general objects to
trigonometric
polynomials goes back at least as far as~\cite[pp.133--134]{Boch},
and also occurs in~\cite{Figa-T}.
\end{remark}

\section{Analysing our method}
\label{sec:fewer}

Theorems~\ref{th:Paley}
and~\ref{th:complementary}
both state
that if~$\hat f$ vanishes on a suitable
part of the complement of
a strongly lacunary set~$K$,
then
\begin{equation}
\label{eq:commonestimate}
\|\hat f\vert K\|_2 \le C\|f\|_1.
\end{equation}
In~\cite[Remark~3]{FouJMAnAp},
an
examination of
the proof in Remark~\ref{rm:ComplementaryProof}
of Theorem~\ref{th:complementary}
revealed that inequality~(\ref{eq:commonestimate})
follows, with~$C=2$,
if~$\hat f(n) = 0$
whenever~$n$ is equal to an alternating sum
\[
k_{j_1} - k_{j_2} + \cdots +
k_{j_{2i-1}} - k_{j_{2i}} + k_{j_{2i+1}},
\]
with at least~$3$ terms and with a strictly increasing index sequence~$(j_\ell)$.
There is no requirement here that~$K$ be strongly lacunary,
or that it be
enumerated
monotonically.

We
examine our
new
proof of Theorem~\ref{th:Paley}
with
a similar
goal.
Given a subset~$D$
of the integer group~$\Z$,
let~$V(D)$ denote the closed subspace of~$L^2(\T)$ spanned by the products~$z^nh$ for which~$n \in D$.
The subspaces~$L_j$
used to prove Paley's theorem
had the form~$V(D_j)$
where~$
D_j = \left\{n: n < -k_j\right\}
$.

For any choice of sets~$D_j$, let~$L_j= V(D_j)$.
Then
\[
A_{j}L_{j} = V(D_{j} + k_{j}),
\quad\text{and}\quad
A_{j+1}L_j = V(D_j + k_{j+1}),
\]
where~$j < J$
in the latter case.
We
required that~$g$
be orthogonal
to the subspace~$A_{j}L_j$
for all~$j > 1$.
By formula~(\ref{eq:coefficients}), this happens if only if
\begin{equation}
\label{eq:hypothesis}
\hat f(n) = 0
\quad\text{
for all
integers~$n$
in the set}\quad
\bigcup_{j=2}^{J} (D_j + k_{j}).
\end{equation}
%
%
Our
proof
uses three
properties
of the subspaces~$L_j$ and their images.
\begin{enumerate}
\item
$A_jh \in A_{j+1}L_j$
when~$j < J$.
\label{it:Membership}
\item
$L_j \supset L_{j+1}$
when~$j < J$.
\item
$A_{j}L_{j-1}
\subset A_{j+1}L_j$
when~$1 < j < J$.
\end{enumerate}
The membership condition~(\ref{it:Membership}) holds if
\begin{equation}
\label{eq:kjInShiftedDj}
k_j \in
D_j
+ k_{j+1}
\end{equation}
when~$j < J$.
The subspaces~$L_j$ and their images~$A_{j+1}L_j$ nest suitably if
\begin{gather}
\label{eq:antinest}
D_j \supset D_{j+1},
\\
\label{eq:imagenest}
\text{and}\quad
D_{j-1}
+ k_j
\subset
D_j + k_{j+1},
\end{gather}
where~$j < J$ in both cases, and~$j > 1$ in the second case.

Extend the finite sequence~$(k_j)_{j=1}^J$ to a
doubly-infinite
sequence,
in the integers or some larger abelian group,
with no monotonicity
or disjointness
requirement,
and seek sets~$D_j$ satisfying the 
three
conditions above for
all values of~$j$.
The lack of
special conditions at
endpoints 
for~$j$
makes it easier to find a pattern that
works.

Form the sets~$G_{j+1} = k_{j+1} + 
D_j$.
Making
them
minimal will do the same for the sets~$D_j$.
The three conditions on
the
latter
hold for all~$j$ if and only if
\begin{gather}
\label{eq:InGj}
k_j \in G_{j+1},\\
\label{eq:ShiftedGj'sAntinest}
G_{j+1} - k_{j+1} \subset G_j - k_j,
\\
\label{eq:Gj'sNest}
\text{and}\quad
G_j \subset G_{j+1}
\end{gather}
for all~$j$.
Rewrite
the second
condition
above
as
\begin{equation}
\label{eq:RewriteShiftedGj'sAntinest}
G_{j+1} - \Delta k_j \subset G_j,
\end{equation}
where~$\Delta k_j = k_{j+1} - k_j$.
Since~$G_j \subset G_{j+1}$, it follows that
\begin{align}
G_{j+1} - 2\Delta k_j
&= (G_{j+1} -\Delta k_j) - \Delta k_j
\notag\\
&\subset G_j - \Delta k_j
\subset G_{j+1} - \Delta k_j
\subset G_j
\notag
\end{align}
Let~$i$
and~$i'$
be
integers
for which~$i < i'$,
and
let~$(m_{j'})_{j' = i}^{i'}$
be
a sequence of
strictly
positive integers.
Iterate
the reasoning above to
get that
\[
G_{i'+1} - \sum_{j'=i}^{i'} m_{j'}\Delta k_{j'}
\subset G_i.
\]
Combine this with condition~\eqref{eq:InGj} to get that
\[
k_{i'} - \sum_{j'=i}^{i'} m_{j'}\Delta k_{j'}
\in G_i.
\]
Since~$k_{i'} - \sum_{j=i}^{i'-1} \Delta k_j = k_i$,
the expression on the left above
is equal to
\begin{equation}
\label{eq:Shiftki}
k_{i} - \sum_{j'=i}^{i'-1} (m_{j'}-1)\Delta k_{j'} - m_{i'}\Delta k_{i'}
=
k_{i} - \sum_{j'=i}^{i'} n_{j'}\Delta k_{j'}
\end{equation}
say,
where~$n_{j'} \ge 0$ for all~$j'$
and~$n_{i'} > 0$.

Condition~\eqref{eq:Gj'sNest}
forces~$G_{j+1}$
to contain
combinations
of the 
form~\eqref{eq:Shiftki}
when~$i \le j+1$.
By
conditions~\eqref{eq:InGj} and~\eqref{eq:Gj'sNest},
it
must also
contain~$k_i$
when~$i \le j$.
%
So~$G_{j+1}$
must contain all
combinations~$k_i - s_i$
in which
\begin{enumerate}
\item{}
$i \le j+1$.
\item{}
$s_i$ is a sum of
finitely many
copies of~$\Delta k_{j'}$ 
in
which~$j' \ge i$.
\item{}
Repetitions are allowed in the sum~$s_i$.
\item{}
That sum
is
nonempty if~$i = j+1$.
\end{enumerate}


Let each
set~$G_{j+1}$
contain nothing else.
Then it is obvious that conditions~\eqref{eq:InGj} and \eqref{eq:Gj'sNest} hold.
For
the remaining
condition~\eqref{eq:RewriteShiftedGj'sAntinest},
suppose that the four
statements
listed
above
hold for~$k_i - s_i$.
In the cases where~$i \le j$,
\[
(k_i - s_i) - \Delta k_{j}
= k_i - (s_i + \Delta k_{j}),
\]
which belongs 
to~$G_i$, and hence
to~$G_j$.
When~$i = j+1$
instead,
\begin{gather}
(k_i - s_i) - \Delta k_{j}
= (k_{j+1} - s_{j+1}) - \Delta k_{j}
\notag
\\
= (k_{j+1} - \Delta k_j) - s_{j+1}
= k_j - s_{j+1},
\notag
\end{gather}
which also belongs to~$G_j$,
since the sum~$s_{j+1}$ is nonempty.

The conclusion that~$\left[\sum_j |\hat f(k_j)|^2\right]^{1/2} \le 2\|f\|_1$ follows if~$\hat f$ vanishes on
all
the 
sets~$D_j + k_j$.
They
coincide with
the difference sets~$G_{j+1} - \Delta k_j$ considered above.
There,
expressions
of the
form~$k_j - s'_j$
arose in two ways,
as~$k_j - (s_j + \Delta k_j)$,
and as~$(k_{j+1} - \Delta k_j) - s_{j+1}$.
In both cases, the
sum~$s'_j$
is
nonempty.
All
nonempty
sums~$s'_j$
of differences~$\Delta k_{j'}$
in which~$j' \ge j$
arise in
these ways.

Call such a
combination~$k_j - s'_j$
a
top member
of the set~$G_{j+1} - \Delta k_j$.
The other members of that set have the form~$k_{j'} - s'_{j'}$ where~$j' < j$
and~$s'_{j'}$ contains a copy
of~$\Delta k_{j'+1}$.
Then~$k_{j'} - s'_{j'}$ is a top member of~$G_{j'+1} - \Delta k_{j'}$.



Denote the union of the
sets~$G_{j+1} - \Delta k_j$,
or their subsets of top members,
by~$\Schur((k_j))$.
It comprises all
combinations~$k_j - s'_j$ as above
where
the sum~$s'_j$ is nonempty.
Rewrite~$k_j - s'_j$
as
\begin{equation}
\label{eq:epsilons}
\sum_{j' } \varepsilon_{j'} k_{j'},
\end{equation}
where the coefficients~$\varepsilon_{j'}$ are integers, and only finitely-many of them differ from~$0$.
Such sums belong to~$\Schur((k_j))$
if and only
if
these coefficients
satisfy the following conditions, which arose in
the dual method in~\cite{FouArk},
and also
arise
in
the
one
used
in~\cite{PS, Smith, Yudin}.
\begin{itemize}
\item{}
\label{it:fullsum}
The full
sum~$\sum_{j'} \varepsilon_{j'}$
is equal to~$1$.
\item{}
All partial sums of
that
full sum are nonnegative.
\item{}
All partial sums after the first positive one are positive.
\item{}
Some partial sum is greater than~$1$.
\end{itemize}
%
%

Specify~$G_{j+1}$ and~$\Schur((k_j))$
in the same way for enumerations
of the form~$(k_j)_{j=I}^\infty$,
where~$I$
is finite, except 
for
requiring
that~$j \ge I$.
Given an enumeration
of form~$(k_j)_{j=-\infty}^J$
or~$(k_j)_{j=I}^J$
where~$J$ is finite,
specify~$G_{j+1}$ as above
when~$j < J$,
and
let~$\Schur\left((k_j)\right)$ be
the union of the sets~$G_{j+1} - \Delta k_j$
for these values of~$j$.
In all cases, this 
union
is
still
the set of
sums~(\ref{eq:epsilons})
with the four properties listed above.

Conditions~\eqref{eq:InGj},
\eqref{eq:Gj'sNest} 
and~\eqref{eq:RewriteShiftedGj'sAntinest}
hold for the same reasons as
before.
%
Let~$D_j$
be~$G_{j+1} - k_{j+1}$
when this
difference
set is defined.
Conditions~(\ref{eq:kjInShiftedDj}), 
(\ref{eq:antinest}) and~(\ref{eq:imagenest})
then
hold
except when~$j = J-1$
and~$J$ is the largest index in the enumeration.
These
cases are not required
in
putting~$L_j  = V(D_j)$
and
applying
the
method in
our
proof of
Theorem~\ref{th:Paley}.
Doing that
yields the
following.

\begin{theorem}\label{th:S(K)}
Let~$K$ be a subset of the group~$\Z$,
and let~$f \in L^1(\T)$.
If~$\hat f$ vanishes on~$\Schur((k_j))$
for some enumeration~$(k_j)$ of~$K$,
then
\begin{equation}\label{eq:S((k_j))}
\|\hat f|K\|_2 \le 2\|f\|_1.
\end{equation}
\end{theorem}

Again, there is no requirement that~$K$ be strongly
lacunary,
or that it be enumerated in increasing order.
In
many
cases,~$\Schur((k_j))$ overlaps with~$K$, and the  hypothesis in the theorem then forces~$\hat f$ to vanish on that overlap.
When~$K$ is strongly lacunary
and enumerated in increasing order,
however,
no such overlap can occur,
because~$\Schur((k_j))$ is then included in the
strictly negative cone.
In
most
cases, that inclusion is strict,
and Theorem~\ref{th:S(K)} sharpens
Theorem~\ref{th:Paley}.

As in
Remark~\ref{rm:To trig polys},
Theorem~\ref{th:S(K)}
extends,
with the same constant~$2$,
to Fourier coefficients of measures.
One
can
also
replace~$\Schur((k_j))$ by a significantly smaller set,
at the cost of using a larger constant
in inequality~(\ref{eq:S((k_j))}).
Let~$S((k_j))$ consist of all integers~$m$ with
at least one representation~\eqref{eq:epsilons}
in which the coefficients~$\varepsilon_{j'}$ belong to the set~$\{-1,
0, 
1\}$
and satisfy the four conditions for membership of~$m$ in~$\Schur((k_j))$.

Consider the corresponding
notion
on
abelian groups.
Recall the definition of the set~$\Riesz(K)$ in Section~\ref{sec:totalize}.
Clearly,
\[
S((\gamma_j))
=
S((\gamma_j))
\cap \Riesz(K).
\]
The
following statement
is proved
in the next section.

\begin{theorem}\label{th:S_1(K)}
Let~$K$ be a subset of
a discrete abelian group with dual~$G$,
and
let~$\mu$ be a regular Borel measure
on~$G$.
If~$\hat \mu$ vanishes
on~$S((\gamma_j))$
for some enumeration~$(\gamma_j)$ of~$K$,
then
\begin{equation}\label{eq:S_1(K)}
\|\hat \mu|K\|_2 \le 4\|\mu\|.
\end{equation}
\end{theorem}


\begin{remark}\label{rm:supplproof}
The two theorems
above were
proved
in the late~$1970$'s in~\cite{FouArk} via a dual construction using the Schur algorithm. That method yielded
inequalities~(\ref{eq:S((k_j))}) and~(\ref{eq:S_1(K)})
with the smaller constants~$\sqrt2$ and~$2\sqrt2$.
The utility of the methods
used in the present paper was
understood by the early~$1970$'s, however, so that the application in~\cite{FouArk} to
``half'' of
the Littlewood conjecture for exponential sums
could have been obtained
somewhat
earlier.
\end{remark}


\begin{remark}
The
dual construction
in~\cite{PS} and~\cite{Smith}
can also be used to prove Theorem~\ref{th:S(K)}, with constant~$2$.
\end{remark}

\begin{remark}
In
the case where
the sequence~$(k_j)$ is doubly
infinite, the
sets~$\Schur((k_j))$,~$G_{j+1}$ and~$D_j$
can also be described using suitable
partial orders or
\emph{preorders}
that are compatible with addition.
For each
index~$j$,
let~$P_j$
be the semigroup generated by the 
differences~$\Delta k_{i}$
in which~$i \ge j$.
Write~$m <_j n$
when~$n - m \in P_j$,
with no requirement
that~$0 \notin P_j$.
Then
\begin{enumerate}
\item{}
\label{G's preorder}
$m \in \Schur((k_j))$ if and only if~$m <_j k_j$ for some~$j$.
\item{}
$m \in G_{j+1}$ if and only if~$m <_{j+1} k_{j+1}$ or~$m \le_i k_i$ for some~$i \le j$.
\item{}
$m \in D_j$ if and only if~$m <_{j+1} 0$ or~$m \le_i k_i - k_{j+1}$ for some~$i \le j$.
\end{enumerate}
\end{remark}

\begin{remark}
\label{rm:RepresentDj}
In the second
part of
the description of~$D_j$ just above, write~$k_{i} - k_{j+1}$ as~$-\sum_{j' = i}^j \Delta k_{j'}$.
It follows that the members of~$D_j$ are 
the combinations~$-\sum_{j'} n_{j'}\Delta k_{j'}$
with integer coefficients~$n_{j'}$ 
having the following properties.
\begin{itemize}
\item{}
$n_{j'} \ge 0$ for all~$j'$.
\item{}
$n_{j'} > 0$ for some~$j'$.
\item{}
The set of indices~$j' < j$
for which~$n_{j'} \ne 0$ has no
gaps,
and contains~$j-1$
unless that set is empty.
\end{itemize}
The antinesting
property~(\ref{eq:antinest})
of the sets~$D_j$
is
then 
easy to check.
\end{remark}

\begin{remark}
\label{rm:PreorderPaley}
So is the fact that
each set~$D_j$
is an additive semigroup.
Define preorders by
saying that~$m < _j^* n$ when~$m - n \in D_j$.
Rewrite
conditions~\eqref{eq:kjInShiftedDj} to~\eqref{eq:imagenest}
as follows.
\begin{description}
\item[\textnormal{Membership}]
$k_j <_{j}^* k_{j+1}$.
\item[\textnormal{Antinesting}]
If~$m < _{j+1}^* n$,
then~$m < _{j}^* n$.
\item[\textnormal{Nesting}]
If~$m < _{j-1}^* k_{j}$,
then~$m < _{j}^* k_{j+1}$.
\end{description}
The hypothesis in Theorem~\ref{th:S(K)} is that~$\hat f$ vanishes on the union of the sets~$D_j + k_j$, that
is~$\hat f(m) = 0$
whenever there is some index~$j$ for
which~$m <_j^* k_j$.
\end{remark}

\section
{Direct
Proof of Theorem~\ref{th:S_1(K)}}
\label{sec:completing}


We work initially with stronger hypotheses.
\begin{lemma}\label{th:measures}
Let~$K$ be a
strongly lacunary
set
in
a
partially ordered
discrete abelian group~$\Gamma$,
and let~$\mu$ be a regular Borel measure
on the dual of~$\Gamma$.
Enumerate~$K$ in increasing order
as~$(\gamma_j)$.
If~$\hat \mu$ vanishes on~$\Schur((\gamma_j))\cap\Riesz(K)$, then
\begin{equation}\label{eq:MeasuresS_1(K)}
\|\hat u|K\|_2 \le 4\|\mu\|.
\end{equation}
\end{lemma}

\begin{proof}
Denote the group dual to~$\Gamma$ by~$G$.
The proof of Theorem~\ref{th:S(K)}
applies to functions in~$L^1(G)$
whose coefficients vanish on~$\Schur((\gamma_j))$.
The
methods
in  Section~\ref{sec:totalize}
then
yield
inequality~(\ref{eq:MeasuresS_1(K)})
when~$\hat \mu$
vanishes
on~$\Schur((\gamma_j))\cap\Riesz(K)$.
\end{proof}

\begin{proof}[Proof of Theorem~\ref{th:S_1(K)}]
Drop
the requirement
that  the finite set~$K$ be strongly 
lacunary.
%
Form the product group~$G \times \T$
and its dual~$\Gamma \times \Z$.
Define a partial order on that dual group
by declaring that
\[
(\gamma', n') < (\gamma, n)
\quad\text{when~$n' < n$.}
 \]
The set~$\tilde K$ of pairs~$(\gamma_j, 3^j)$
is strongly lacunary relative
to this partial order.
Note that
if~$(\gamma, 
n) \in S(\tilde K)$, then~$\gamma \in S(K)$.

Identify~$\T$ with the interval~$(-\pi, \pi]$
with addition modulo~$2\pi$.
Identify~$G$
with the subgroup~$G \times \{0\}$
of~$G \times \T$.
Given a measure~$\mu$
on~$G$ form a measure~$\tilde \mu$
on~$G \times \T$ by first transferring~$\mu$
to~$G \times\{0\}$,
and then extending it to vanish outside
that subgroup of~$G\times\T$. Note that~$\|\tilde \mu\| = \|\mu\|$, and that
\[
\hat{\tilde \mu}(\gamma,
n)
= \hat \mu(\gamma)
\]
in all cases.

Suppose that~$\hat \mu$ vanishes on~$S(K)$.
Then~$\hat{\tilde \mu}$ vanishes
on~$S(\tilde K)$. Since~$\tilde K$
is strongly lacunary,
Lemma~\ref{th:measures}
applies to~$\tilde \mu$, and yields
that
\[
\|\hat \mu|K\|_2 = \|\hat{\tilde \mu}|\tilde K\|_2
\le 4\|\tilde \mu\| = 4\|\mu\|.\qedhere
\]
\end{proof}

\begin{remark}
\label{rm:Drury}

The idea of adding
one dimension to remove some unwanted frequencies goes back at least as far as~\cite{Dru}.
\end{remark}


\begin{appendix}
\begin{section}{
Other nestings
}
\label{sec:compare}

For the classical Paley theorem, the authors of~\cite{KP}
used analytic
factorization and
projections into finite-dimensional subspaces.
A version of
their argument,
without
analytic factors,
runs as follows.

Factor~$f$ as
before,
and form
the subspaces~$L_j$.
%
As in Section~\ref{sec:fewer},
given any set~$S$ of integer, let~$V(S)$ be the closure
in~$L^2$
of the span of the products~$z^nh$ for which~$n \in S$.
Also denote
the
subspaces~$V(Z\cap(\infty, 0))$ 
and~$V(Z\cap(\infty, 0])$ 
by~$M$ and~$M''$ respectively.
When~$1 \le j \le J$,
let~$M''_j$ be the part
of~$M''$
that is orthogonal
to~$L_j$;
let~$M''_0$
be the trivial subspace.
%
More generally,
denote the part
of~$V(\Z\cap(-\infty, b])$
that is orthogonal to~$V(\Z\cap(-\infty, a))$
by~$W\{a)^\perp,b]\}$,
and denote the corresponding part
of~$V(\Z\cap(-\infty, b))$
by~$W\{a)^\perp,b)\}$.
Then~$M''_j = W\{-k_j)^\perp,0]\}$.

Like
the subspaces~$M_j$
in Remark~\ref{rm:ComplementaryProof},
the
subspaces~$M''_j$
are finite-dimensional, and
form an increasing
nest.
The shifted
subspaces~$A_jM''$ and~$A_jM''_j$
are equal to~$V(Z\cap(\infty, k_j])$
and~$W\{0)^\perp,k_j]\}$
respectively.
In
particular,~$A_jh \in A_jM''$.
%

Denote the orthogonal projection
onto~$A_jM''_j$
by~$Q''_j$.
Since~$\hat g$ vanishes on the negative 
integers,~$g \perp M$.
Split~$A_jh$ as~$u + v$,
where~$u \in M$
and~$v \in A_jM''_j$.  Then
\[
(g, A_jh) = (g, u) + (g, v)
= (g, v)
= (g, Q''_jA_jh)
= (Q''_jg, A_jh).
\]
Much as
in Remark~\ref{rm:ComplementaryProof}, 
write
this as~$a''_j + b''_j$, where
\[
a''_j =
\ (\{Q''_{j}-Q''_{j-1}\}g, A_jh),
\quad\text{and}\quad
b''_j = (g, Q''_{j-1}A_jh),
\]
with the convention that~$Q''_0 = 0$.
Estimate~$\|a''\|_2$ as before.

When~$j > 1$, the
range of~$Q''_{j-1}$
is~$W\{0)^\perp,k_{j-1}]\}$,
which
is the image under~$A_j$
of~$W\{-k_j)^\perp,k_{j-1}-k_j]\}$.
Denote the
orthogonal
projection onto
the latter
subspace
by~$R''_j$.
Then~$b_j'' = (g, A_jR''_{j}h)$.

Now~$W\{-k_j)^\perp,k_{j-1}-k_j]\}$
is included 
in~$M''_j$.
By strong lacunarity,
it is also
included in~$L_{j-1}$,
and hence is orthogonal
to~$M''_{j-1}$.
The orthogonal projections~$P''_j$ and~$P''_{j-1}$ onto~$M''_j$ and~$M''_{j-1}$
therefore have the properties that~$R''_jP''_j = R''_j$ and~$R''_jP''_{j-1} = 0$.
So
\[
A_jR''_jh = A_jR''_j(P''_j-P''_{j-1})h,
\]
and~$\|b''\|_2 \le \|g\|_2\|h\|_2$.

As in~\cite{KP}, simpler choices
work
when~$g$ and~$\overline h$
are analytic.
Replace
the subspaces~$L_j$
with
the closures in~$L^2$ of the
spans
of the functions~$z^n$ for which~$n < -k_j$.
Let~$M''$ be
the closure in~$L^2$ of the
span
of the functions~$z^n$ for which~$n \le 0$.
Form the orthogonal
complements~$M''_j$
in~$M''$,
and estimate as above.

\label{rm:Part of L_j}

In general, one can also use the orthogonal complements
of each~$L_{j+1}$
in each~$L_j$,
that is~$W\{-k_{j+1})^\perp, -k_j)\}$.
These subspaces are not nested, but their
images~$A_{j+1}W\{-k_{j+1})^\perp, -k_j)\}$
are, because they
coincide
with the spaces~$W\{0)^\perp,k_{j+1}-k_j)\}$.
%

Denote the projection
onto
the latter
by~$Q'''_j$.
Much as
above,
the facts that~$A_jh \in V(\Z\cap(-\infty, k_{j+1}-k_j))$,
and~$g \perp V(\Z\cap(-\infty, 0))$
make~$(g, A_jh)$ equal to~$(Q'''_{j}g, A_jh)$.
In turn, that splits as~$a'''_j + b'''_j$,
where~$a'''_j = ((Q'''_j - Q'''_{j-1})g, A_jh)$,
and~$b'''_j = (g, Q'''_{j-1}A_jh)$,
again with the convention
that~$Q'''_0 = 0$.

When~$j > 1$,
denote
the projection
onto~$W\{-k_{j})^\perp, -k_{j-1})\}$ by~$P'''_{j-1}$.
Then~$b'''_j = (g, A_jP'''_{j-1}h)$.
The ranges~$W\{-k_{j})^\perp, -k_{j-1})\}$ of 
the
various
projections~$P'''_{j-1}$
are orthogonal,
because the corresponding
intervals~$[-k_{j}, -k_{j-1})$ are disjoint.
It follows that~$\|b'''\|_2 \le \|g\|_2\|h\|_2$.

\end{section}
 \end{appendix}

\bibliographystyle{amsplain}

\end{document}